\documentclass[a4paper,oneside,11pt]{article}

\usepackage{geometry}
\usepackage{fancyhdr}
\usepackage{lipsum}
\usepackage{amsmath,amsfonts,amscd,amssymb,amsthm, mathrsfs}
\usepackage{dsfont}
\usepackage{longtable}
\usepackage[english]{babel}
\usepackage[utf8]{inputenc}
\usepackage[active]{srcltx}
\usepackage[T1]{fontenc}
\usepackage{graphicx}
\usepackage{enumitem}
\usepackage{bbm}
\usepackage{enumerate}   
\usepackage{mathtools}
\usepackage[colorlinks = true,
            linkcolor = blue,
            urlcolor  = magenta,
            citecolor = blue,
            anchorcolor = blue]{hyperref}
\usepackage{subfig}
\usepackage{MnSymbol}
\usepackage{stmaryrd}
\usepackage{nicefrac}
\usepackage{natbib}
 \usepackage{rotating}
\usepackage{xcolor}
\usepackage{framed}
\usepackage[affil-it]{authblk}
\usepackage{verbatim}

\colorlet{shadecolor}{blue!15}

\newtheorem{theorem}{Theorem}%[section]

\newtheorem{proposition}{Proposition}

\newtheorem{lemma}{Lemma}

\newcommand {\tr}{\mathrm{tr}}
\newcommand {\R}{\mathbb{R}}
\newcommand{\PP}{\mathbb{P}}
\newcommand{\E}{\mathbb{E}}%\boldsymbol{E}
\renewcommand {\H}{\mathbb{H}}

\newcommand {\T}{\mathcal{T}}
\renewcommand {\P}{\mathcal{P}}

\newcommand {\BB}{\mathscr{B}}

\newcommand {\CC}{\mathscr{C}}
\newcommand {\A}{\mathscr{A}}

\usepackage{enumitem, hyperref}
\makeatletter
\def\namedlabel#1#2{\begingroup
  #2%
  \def\@currentlabel{#2}%
  \phantomsection\label{#1}\endgroup
}
\makeatother

 \title{\Large{A Karhunen–Loève Theorem for Random Flows in Hilbert spaces}}

\author{
  Leonardo V.\ Santoro
  \qquad
  Kartik G. Waghmare
  \qquad
  Victor M. Panaretos   \\ {\footnotesize{\texttt{leonardo.santoro@epfl.ch} \qquad\,\,\texttt{kartik.waghmare@epfl.ch} \qquad\,\,\texttt{victor.panaretos@epfl.ch}}}}
  
  \affil{Institut de Math\'ematiques\\École Polytechnique Fédérale de Lausanne}

\date{\today}

\begin{document}

\maketitle
\begin{abstract}
    We develop a generalisation of Mercer's theorem to operator-valued kernels in infinite dimensional Hilbert spaces. We then apply our result to deduce a Karhunen-Loève theorem, valid for mean-square continuous Hilbertian functional data, i.e.\ flows in Hilbert spaces. That is, we prove a series expansion with uncorrelated coefficients for square-integrable random flows in a Hilbert space, that holds uniformly over time.

    \bigskip
      \noindent \textbf{MSC2020 classes:} 60G12, 62R10\\ %62H25,  46C99, 15A18 \\
    \textbf{Key words:} Mercer's theorem, Functional Principal Components, Hilbertian functional data, random series expansion

\end{abstract}

\setcounter{section}{0}

\section{Introduction}

The Karhunen-Lo\`eve theorem is a fundamental result on stochastic processes, playing a central role in their probabilistic construction, numerical analysis and statistical inference \citep{adler2010geometry,le2010spectral,hsing2015theoretical}. In its simplest form, it provides a countable decomposition of a time-indexed real-valued random function into a ``bi-orthogonal'' Fourier series that separates its stochastic from its functional components: the basis functions are deterministic orthogonal functions, and their coefficients are uncorrelated random variables. When the basis functions are ordered by decreasing coefficient variance, the $k$-truncated expansion provides the best $k$-dimensional approximation of the process in mean square. Importantly, for random functions that are mean-square continuous, the series can be interpreted \emph{pointwise}. These properties explain the catalytic role the expansion has played (and continues to play) in
% the development of 
the field of Functional Data Analysis \citep{wang2016functional} -- in fact, the field's very origin traces to Grenander's use of the expansion as a coordinate system for inference on random functions \citep{grenander1950stochastic}. 

Traditionally, functional data analysis focused on real-valued functions defined on a compact interval. Increasingly, though, the complexity of functional data escapes this context. Modern functional data can come in the form of functional flows (random maps from an interval into a function space) or more generally functional random fields (random maps from a Euclidean set into a function space). With the aim of making such data amenable to the tools of functional data analysis, we establish a generalisation of the Karhunen-Lo\`eve expansion to random flows (and more generally random fields) valued in an abstract separable Hilbert space, possibly infinite dimensional (Theorem \ref{kl-ours}). To do so, we first establish a version of Mercer's theorem for non-negative definite kernels valued in separable Hilbert spaces (Theorem \ref{thm : mercers, ours}).

\section{Background and Problem Statement}

Let $\H$ be a separable Hilbert space with inner product $\langle\cdot,\cdot\rangle_{\H} \;:\; \H\times\H\rightarrow\R$ and induced norm $\|\cdot \|_{\H}\;:\; \H \rightarrow \R_{+}$, with $\mathrm{dim}(\H)\in \mathbb{N}\cup\{\infty\}.$
We denote by $\BB(\H)$, $\BB_1(\H)$ and $\BB_2(\H)$ the space of
bounded, trace-class (TC) and Hilbert-Schmidt (HS) linear operators on $\H$, respectively, with corresponding norms:
 $$
 \|B\|_{\BB(\H)}:=\sup_{h\in\H, \|h\|=1}\|B h\|,\quad 
\|B\|_{\BB_{1}(\H)}:= \tr(\sqrt{B^*B}), \quad 
 \|B\|_{\BB_{2}(\H)}:= \sqrt{\tr(B^*B)},
 $$
 where the adjoint $B^*$ of a linear operator $B\in \mathscr{B}(\H)$ is defined via $\langle Bu,v\rangle_{\H} = \langle u,B^*v\rangle_{\H}$ for all $u,v \in \H$. Note that
 $
  \|B\|_{\BB(\H)} \leq
 \|B\|_{\BB_2(\H)}\leq
 \|B\|_{\BB_1
 (\H)}.
 $
 For $u,v \in \H$, the operator $u\otimes_{\H}v \;:\; \H\rightarrow \H$ defined by $ ( u\otimes_{\H}v ) h = \langle u,h\rangle_{\H} v$ is bounded and linear. 

We say that a bounded operator $B$ is compact if for any bounded sequence $\{h_n\}_{n\geq 1}$ in $\H$, $\{B h_n\}_{n\geq 1}$ contains a convergent subsequence. Let $B\in \BB(\H)$ be compact and self adjoint. Denote by $\{e_j\}_{j\geq 1}$ its eigenvectors, with corresponding eigenvalues $\{\lambda_j\}_{j\geq 1}$, ordered so that $|\lambda_1|\geq|\lambda_2|\geq\dots$. Then
$\{e_j\}_{j\geq 1}$ comprises a Complete Orthonormal System (CONS) for $\overline{\mathrm{Im}(B)}$ and we may write
$
B = \sum_{j\geq 1} \lambda_j e_j\otimes_{\H} e_j.
$

\bigskip

Let $\T$ be any compact subset of a Euclidean space. We denote by $L^2(\T, \H)$ the Hilbert space of square integrable $\H$-valued functions $f:\T \rightarrow \H$, i.e., $$L^{2}(\T,\H)=\left\{f: \int_{\T}\|f(t)\|^{2}_{\H} d t<\infty\right\}.$$ 
The associated inner product $\langle\cdot, \cdot\rangle_{L^{2}(\T,\H)}$ is defined by $\langle f, g\rangle_{L^{2}(\T,\H)}=\int_{\T}\langle f(t), g(t)\rangle_{\H} d t$ with corresponding norm $\|\cdot\|_{L^{2}(\T,\H)}$ by $\|f\|_{L^{2}(\T,\H)}^{2}=\langle f, f\rangle_{L^{2}(\T,\H)}$.

Let $\chi\in {L^{2}(\T,\H)}$ be a (mean zero) random flow in $\H$ with finite second moment, $\E \| \chi\|_{L^2(\T,\H)}^2<\infty$. We will refer to $\chi$ as \textit{Hilbertian flow}, as in \cite{kim2020principal}, though it could also be a termed a \emph{Hilbertian field} when $\mathrm{dim}(\mathcal{T})\geq 2$. Denote by $\CC\in \BB(L^{2}(\T,\H))$ its covariance operator:
\begin{equation}\label{eq : log-process covariance}
\CC\;:\;L^{2}(\T,\H) \rightarrow L^{2}(\T,\H),
\qquad \quad
\CC  = \E \left[ 
\chi \otimes_{L^{2}(\T,\H)}
\chi
\right]
\end{equation}
or, equivalently:
$$
\langle\CC U, V\rangle_{L^{2}(\T,\H)}:=
\E\left[
\langle\chi, U\rangle_{L^{2}(\T,\H)}
\langle \chi, V\rangle_{L^{2}(\T,\H)}
\right], \quad \text { for } U, V \in L^{2}(\T,\H)    
$$
Note that $\CC$ is nonnegative-definite and trace-class. In particular, $\CC$ is compact and self-adjoint; therefore, it admits the following spectral decomposition in terms of its eigenvalue-eigenfuction pairs (e.g. \citet[][Theorem 7.2.6]{hsing2015theoretical}): 
$$
\CC=\sum_{k=1}^{\infty} \lambda_{k} {\Phi}_{k} \otimes_{L^{2}(\T,\H)} {\Phi}_{k},
$$
where $\lambda_{1}>\lambda_{2}>\cdots>0$ are the eigenvalues and ${\Phi}_{k}\:  : L^{2}(\T,\H)  \: \rightarrow  L^{2}(\T,\H)$ the corresponding eigenfunctions for $\CC$, forming a complete orthonormal system.

\smallskip
\noindent Consequently, $\chi$ admits the following decomposition
\begin{equation}\label{L2expansion}
\chi =\sum_{k=1}^{\infty} \langle \chi, {\Phi}_{k}\rangle_{L^{2}(\T,\H)} \boldsymbol{\Phi}_{k}
\end{equation}
where the convergence is understood in the mean $L^2(\T,\H)$ norm sense
\begin{equation} \label{eq : integral KL}
\lim_{N\rightarrow \infty} 
\E \left\| \chi - \sum_{k=1}^{N} \langle \chi , {\Phi}_{k}\rangle_{L^{2}(\T,\H)} {\Phi}_{k}\right\|_{L^2(\T,\H)}^2 = 0 
\end{equation}
and where it may furthermore be shown that $\{ \langle \log _{\Xi}, {\Phi}_{k}\rangle_{L^{2}(\T,\H)}\}_{k\geq 1}$ are uncorrelated random variables with zero mean.
In particular, by \citet[][Theorem 7.2.8]{hsing2015theoretical}, the above decomposition is optimal: for any $N<\infty$ and any CONS $\{\Psi_k\}_{k\geq 1}$ of $L^{2}(\T,\H)$, 
$$\E \left\| \chi - \sum_{k=1}^{N} \langle \chi , {\Phi}_{k}\rangle_{L^{2}(\T,\H)} {\Phi}_{k}\right\|_{L^2(\T,\H)}^2 \quad\leq \quad\E \left\| \chi - \sum_{k=1}^{N} \langle \chi , {\Psi}_{k}\rangle_{L^{2}(\T,\H)} {\Psi}_{k}\right\|_{L^2(\T,\H)}^2.$$

\bigskip
\noindent We ask the following questions:
\begin{center}
    \emph{Is expansion \eqref{L2expansion} interpretable pointwise in $t\in\T$? }\\
    \emph{Is the convergence \eqref{eq : integral KL} valid uniformly over $t\in\T$?}
\end{center}
When $\H = \R$, and assuming mean-square continuity of $\chi$, these two questions have been long known to admit a positive answer in the form of the celebrated Karhunen-Lo\`eve theorem (\cite{karhunen1946spektraltheorie}, \cite{loeve1948functions}; see also \cite{kac1947explicit}), whose proof fundamentally relies on Mercer's theorem on the decomposition of real-valued kernels \citep{mercer1909xvi}.  Extensions to random fields valued in $d$-dimensional Euclidean space $\H = \R^d$ have also been tackled \citep{withers1974mercer}, but the general, infinite-dimensional case has not been addressed, and does \emph{not} straightforwardly follow from the case $\H = \R^d$.

\medskip

\section{Mercer's Theorem for Operator-Valued Kernels}
Consider a function $K\;:\; \T\times \T \rightarrow \BB_1(\H)$. We refer to such a function $K$ as an \textit{operator-valued kernel}.
We say that $K$ is continuous if:
$$
\| K(s,t) - K(s',t') \|_{\BB_1(\H)} \rightarrow 0, \quad \text{as } t'\rightarrow t, s'\rightarrow s,
$$
for every $s,t\in\T$.
We say that $K$ is symmetric if $K(s,t) = K(t,s)^*$. Finally, we say that $K$ is non-negative definite if for every $n\geq 1$ and sequences $\{v_j\}_{j=1,\dots,n}\subset \T$, $\{h_j\}_{j=1,\dots,n} \subset \H$:
\begin{equation}\label{eq : positive kernel}
\sum_{i,j=1}^n \langle K(v_i,v_j) h_i, h_j\rangle_{\H} \geq 0.
\end{equation}

Given an operator-valued kernel, we may define an \textit{integral operator} $\A_K\;:\;L^2(\T,\H)\rightarrow L^2(\T,\H)$ by:
\begin{equation}\label{eq : integral operator}
    \A_K(f)(t) = \int_{\T} K(t,u)f(u) du
\end{equation}
where the integral in \eqref{eq : integral operator} is to be understood as a Bochner integral.

\begin{lemma}\label{lemma : basic kernel properties}
Let $K$ be a continuous kernel.
\begin{itemize}
    \item[(i)] $\A_K$ is compact
    \item[(ii)] if $K$ is symmetric, then $\A_K$ is self-adjoint.
    \item[(iii)] $K$ is non-negative definite if and only if $\A_K$ is non-negative definite.
\end{itemize}
\end{lemma}
\begin{proof}
\textit{(i)} 
Let $\{e_i\}_{i\geq 1}$ be a CONS for $\H$. Then $\P_n := \sum_{i=1}^n e_i \otimes e_i$ converges strongly to the identity. For $n\geq 1$, let $K_n(\cdot,\cdot) := \P_n K(\cdot,\cdot) \P_n$. Note that $\A_{K_n}$ is compact, being of finite rank. To prove that $\A_K$ is compact it thus suffices to show that $\A_{K_n} \rightarrow \A_{K}$ strongly, as $n\rightarrow \infty$. Now:
\begin{align*}
    \| (\A_{K_n} - \A_K) f \|_{L^2(\T,\H)}
    \: &= \:   \left(\int_{\T} \|(K_n(t,u) - K(t,u))f(u)\|_{\H}^2 du \right)^2 \\
    \: &\leq \:  \sup_{t,s} \|(K_n(t,s) - K(t,s))\|_{\BB(\H)} \cdot \|f\|_{L^2(\T,\H)}.
\end{align*}
By continuity, to prove compactness of $\A_K$ it thus suffices to show that:
$$\|(K_n(t,t') - K(t,s))\|_{\BB(\H)}\rightarrow 0, \quad \text{for all } t,s\in T.
$$
Notice that $P_{n}K(s, t)P_{n} \to K(s, t)$ strongly. Because $K(s,t)$ is Hilbert-Schmidt, the convergence also holds in operator norm. Indeed for $g, h \in \H$ we can write
\begin{eqnarray*}
    &|\langle g, (K(s,t) - P_{n}K(s, t)P_{n})h \rangle_{\H}|
    &\leq \sum_{i, j > n} |\langle K(s, t)e_{i}, e_{j} \rangle_{\H}| \cdot |\langle h, e_{i} \rangle_{\H}| \cdot |\langle g, e_{j} \rangle_{\H}| \\
    &&\leq \left[ \sum_{i, j > n} |\langle K(s, t)e_{i}, e_{j} \rangle_{\H}|^{2} \right]^{1/2} \left[ \sum_{i, j > n} |\langle h, e_{i} \rangle_{\H}|^{2} |\langle g, e_{j} \rangle_{\H}|^{2} \right]^{1/2}  \\
    &&\leq \left[ \sum_{i, j > n} |\langle K(s, t)e_{i}, e_{j} \rangle_{\H}|^{2} \right]^{1/2} \|g\|_{\H}\|h\|_{\H}  \\
\end{eqnarray*}
and the conclusion follows.

\medskip

\textit{(ii)} If $K$ is symmetric, for any $f,g\in L^2(\T,\H)$:
\begin{align*}
    \langle \A_K f, g \rangle_{L^2(\T,\H)} 
    \: & = \: \int_{\T} \langle \int_{\T} K(t,u) f(u) du, g(t)\rangle_\H dt\\ 
    \: & = \: \int_{\T} \int_{\T} \langle  K(t,u) f(u), g(t)\rangle_\H dt du\\ 
    \: & = \: \int_{\T} \int_{\T} \langle  f(u), K(t,u)^*g(t)\rangle_\H dt du\\
    \: & = \: \int_{\T} \int_{\T} \langle  f(u), K(u,t)g(t)dt \rangle_\H du\\
    \: & = \: \langle  f, \A_K g \rangle_{L^2(\T,\H)} 
\end{align*}
thus proving that $\A_K$ is self-adjoint.

\medskip

\textit{(iii)} We follow \citet[Theorem 4.6.4][]{hsing2015theoretical}.
 Given $n>0$ let $\delta_n$ be chosen so that $\left|K\left(s_2, t_2\right)-K\left(s_1, t_1\right)\right|<n^{-1}$ whenever $d\left(\left(s_1, t_1\right),\left(s_2, t_2\right)\right)<\delta_n$. As $\T$ is a compact metric space, there exists a finite partition $\left\{E_{n i}\right\}$ of $\T$ such that each $E_{n i}$ has diameter less than $\delta_n$. Let $v_i$ be an arbitrary point of $E_{n i}$ and, for all $(s, t) \in E_{n i} \times E_{n j}$, define $K_n(s, t)$ to be $K\left(v_i, v_j\right)$. The (uniform) continuity of $K$ now has the consequence that
$$
\max _{(s, t) \in \T \times \T}\left|K(s, t)-K_n(s, t)\right|<n^{-1} .
$$
Now let $\A_{K_n}$ be the integral operator with kernel $K_n$. With this choice, we find that, for any $f \in L^2(\T,\H)$,
$$
\left|\langle\mathscr{K} f, f\rangle_{L^2(\T,\H)}-\left\langle\mathscr{K}_n f, f\right\rangle_{L^2(\T,\H)}\right| \leq n^{-1}\|f\|_{L^2(\T,\H)}^2
$$
and
$$
\left\langle\mathscr{K}_n f, f\right\rangle_{L^2(\T,\H)}
=\sum_{i,j=1}^n \left\langle K\left(v_i, v_j\right) \int_{E_{n i}} f(t) dt,  \int_{E_{n j}} f(t) dt \: \right\rangle_{\H}.
$$
which, by \eqref{eq : positive kernel}, proves the positiveness of $\A_K$.

Conversely, suppose that
$$
\sum_{i,j=1}^n \langle K(v_i,v_j) h_i, h_j\rangle_{\H} < 0.
$$
for some $\{v_j\}_{j=1,\dots,n}\subset \T$, $\{f_j\}_{j=1,\dots,n} \subset \H$.
As $K$ is uniformly continuous, there exist measurable disjoint sets $E_1, \ldots, E_m \subset \T$ with $\lvert E_i\rvert >0, v_i \in E_i$ for all $i$ such that:
$$
\max_{\tilde v_i \in E_i} \sum_{i,j=1}^n \langle K(\tilde v_i,\tilde v_j) h_i, h_j\rangle_{\H} < 0.
$$
This implies that:
$$
\sum_{i,j=1}^n \int_{E_i}\int_{E_j} \langle K(u,v) h_i, h_j\rangle_{\H} \,du\, dv < 0
$$
due to the mean-value theorem. Upon observing that the last expression is simply $\langle\A_K f, f\rangle_{\H}$ for $f(\cdot) = \sum_{i=1}^n  (\lvert E_i \rvert)^{-1} h_i \cdot \mathds{1}_{E_i} (\cdot)$, we conclude that $\A_K$ is also not non-negative definite.
\end{proof}

In particular, if $K$ is a continuous, symmetric and non-negative definite kernel,  $\A_K$ admits a spectral decomposition in terms of its eigenvalue-eigenfunction pairs. That is, the eigenfunctions for $\A_K$, $\{\Phi_j\}_{j\geq 1}$, form a CONS for $L^2(\T,\H)$, and if $\{\lambda_j\}_{j\geq 1}$ denote the corresponding eigenvalues, one may write:
\begin{equation}\label{eq : spectral decomp of integral operator}
\A_K = \sum_{j\geq 1} \lambda_j \Phi_j \otimes_{L^2(\T,\H)}\Phi_j.
\end{equation}
Furthermore, the following lemma establishes that the eigenfunctions of $\A_K$ are uniformly continuous.

\begin{lemma}\label{lemma : integral operators maps to continous functions}
Let $K$ be a continous kernel. For each $f \in L^2(\T,\H)$, $(\A_K f)(\cdot)$ is uniformly continuous.
\end{lemma}
\begin{proof}
By compactness of $\T$ and uniform continuity of $K$, for any given $\epsilon>0$ there exists $\delta>0$ such that $\left\| K\left(u, s\right)-K\left(u, t\right)\right \|_{\BB_1(\H)} <\epsilon$ for all $u, s, t \in \T$ with $\left|s-t\right|<\delta$. Then:
\begin{align*}
    \|(\A_K f)(s) -(\A_K f)(t) \|_{\H} =  
    \|\int_{\T} \big( K\left(s, u\right) - K\left(t, u\right) \big) f(u) du \|_{\H}
    \leq \epsilon\|f\|_{L^2(\T,\H)}. 
\end{align*}

\end{proof}

The following lemma will be instrumental in the proof of our generalisation to Mercer's theorem.

\begin{lemma}\label{lemma : pre-mercers}
\begin{itemize}
        \item[(a)] For any $t \in \T$,
    \begin{equation}\label{eq : dominated sequence}
    \sum_{n=1}^N \lambda_j \Phi_j(t)\otimes_\H\Phi_j(t) \leq K\left(t, t\right)
    \end{equation}

    \item[(b)] Let $\{e_i\}_{i\geq 1}$ be a CONS for $\H$. For any $s,t \in \T$,
    $$
   \left\| \sum_{n=1}^N  \lambda_j \Phi_j(s)\otimes_\H\Phi_j(t) \right\|_{\BB_1(\H)} \leq \|K(s,s)\|_{\BB_1(H)}^{1/2}\cdot \|K(t,t)\|_{\BB_1(\H)}^{1/2}
    $$
    
    \item[(c)] the operator $\sum_{n=1}^\infty \lambda_j \Phi_j(s)\otimes_\H\Phi_j(t)$ is well-defined and uniformly continuous wrt to $\|\cdot\|_{\BB_1(\H)}$. Furthermore, the sum converges uniformly.
    
\end{itemize}
\end{lemma}
\begin{proof}
\textit{(a)} 
Let
$$
K_n(s, t)=K(s, t)-\sum_{j=1}^n \lambda_j \Phi_j(s) \otimes_{\H} \Phi_j(t)
$$
and take $\A_{K_n}$ to be the integral operator with kernel $K_n$. Note that $K_n$ is continuous, by continuity of $K$ and by Lemma \ref{lemma : integral operators maps to continous functions}. For any $f\in L^2(\T,\H)$:
$$
\left\langle \A_{K_n} f, f\right\rangle_{L^2(\T,\H)}=\langle\A_{K} f, f\rangle_{L^2(\T,\H)}-\sum_{j=1}^n \lambda_j\left\langle f, \Phi_j\right\rangle_{L^2(\T,\H)}^2=\sum_{j=n+1}^{\infty} \lambda_j\left\langle f, \Phi_j\right\rangle_{L^2(\T,\H)}^2 \geq 0
$$
and $\A_{K_n}$ must be non-negative definite. This implies, by Lemma \ref{lemma : basic kernel properties} (iii), that $K_n$ is non-negative definite and hence $K_n(t, t) \geq 0$, thereby proving \eqref{eq : dominated sequence}.

\medskip

\textit{(b)} Let $\{e_i\}_{i\geq 1}$ be a CONS for $\H$. First, note that:
$$
\|K(t,t) \|_{\BB_1(\H)}  \geq \| \sum_{j=1}^N \lambda_j \Phi_j(t)\otimes \Phi_j(t) \|_{\BB_1(\H)} = \sum_{i\geq 1}\sum_{j=1}^N \lambda_j \langle \Phi_j(t), e_i\rangle_{\H} ^2.
$$
Then, by Cauchy-Schwarz, we obtain that:
\begin{align*}
    \left\|  \sum_{n=1}^N \lambda_j \Phi_j(s)\otimes_\H\Phi_j(t) \right\|_{\BB_1(\H)} 
    \: &=\: \sum_{i\geq 1}\sum_{j=1}^N 
    \lambda_j \langle \Phi_j(t), e_i\rangle_{\H}\langle \Phi_j(s), e_i\rangle_{\H}\\
    \: &\leq \: \sum_{j=1}^N\left(\sum_{i\geq 1} \lambda_j \langle \Phi_j(t), e_i\rangle_{\H} ^2 \right)^{1/2} \cdot \left(\sum_{i\geq 1} \lambda_j \langle  \Phi_j(s), e_i\rangle_{\H} ^2 \right)^{1/2} \\
    \: &\leq \: \|K(t,t) \|_{\BB_1(\H)}^{1/2} \|K(s,s) \|_{\BB_1(\H)}^{1/2}.
\end{align*}

\medskip

\textit{(c)} Fix $\varepsilon>0$. By \textit{(b)} we may conclude that there exists $N_{\varepsilon}$ such that:
$$
\sup_{s,t\in\T}\left\| \sum_{n> N_{\varepsilon}} \lambda_j \Phi_j(s)\otimes_\H\Phi_j(t)\right\|_{\BB_1(\H)} \leq \varepsilon.
$$
Furthermore, for any $N$, uniform continuity of the $\Phi_j(\cdot)$ entails the existence of $\delta>0$ such that:
$$
\left\| \sum_{n\leq N} \lambda_j \Phi_j(s)\otimes_\H\Phi_j(t)
 - \sum_{n\leq N} \lambda_j \Phi_j(s')\otimes_\H\Phi_j(t') \right\|_{\BB_1(\H)} \leq \varepsilon.
$$
whenever $|s-s'| + |t-t'| < \delta$. By the last two displays, it is then clear that there exists $\delta$ such that:
$$
\left\| \sum_{n\geq 1} \lambda_j \Phi_j(s)\otimes_\H\Phi_j(t)
 - \sum_{n\geq 1} \lambda_j \Phi_j(s')\otimes_\H\Phi_j(t') \right\|_{\BB_1(\H)} \leq \varepsilon.
$$
 whenever $|s-s'| + |t-t'| < \delta$.
\end{proof}

\begin{theorem}[Mercer's Theorem for Operator-Valued Kernels]
\label{thm : mercers, ours}
Let $K$ be a continuous, symmetric and non-negative definite kernel, and denote by $\A_K$ the corresponding integral operator. Let $\left\{\lambda_j, \Phi_j\right\}_{j\geq 1}$ be the eigenvalue and eigenfunction pairs of $\A_K$, where $\lambda_j\in\R$ and $\Phi_j\in L^2(\T,\H)$. Then:
$$
K(s, t)=\sum_{j=1}^{\infty} \lambda_j \Phi_j(s)\otimes_{\H} \Phi_j(t),
$$
for all $s, t$, with the series converging absolutely and uniformly.
\end{theorem}

\begin{proof}

 For different continuous kernels $K_1(s, t)$ and $K_2(s, t)$, it is straightforward to construct a function $f$ such that $\int_\T K_2(s, t) f(s) ds$ and $\int_\T K_1(s, t) f(s) ds$ differ. Thus, $K$ is the unique operator kernel that defines $\A_{K}$. Now, the integral operator with the continuous kernel $\sum_{j=1}^{\infty} \lambda_j \Phi_j(s) \otimes \Phi_j(t)$ has the same eigen-decomposition as $\A_K$ and is therefore the same operator. Thus, $K(s, t)=\sum_{j=1}^{\infty} \lambda_j \Phi_j(s) \otimes \Phi_j(t)$ for all $s, t\in \T$ with the right-hand side converging absolutely and uniformly as a consequence of Lemma \ref{lemma : pre-mercers}.

\end{proof}

\noindent Finally, we can now see that the integral operator $\A_K$ in \eqref{eq : integral operator} is trace class. 
\begin{proposition}
Let the continuous kernel $K$ be symmetric and non-negative definite and $\A_K$ the corresponding integral operator. Then:
$$
\|\A_{K}\|_{\BB_1(L^2(\T,\H))}=\int_\T \|K(s, s)\|_{\BB_1(\H))} ds.
$$
\end{proposition}
\begin{proof}
We see that, for any CONS $\{e_i\}_{i\geq 1}$ of $\H$:
\begin{align*}
    \tr\{\A_K\} \:=&\: \sum_{j\geq 1} \lambda_j
     \:=\: \sum_{j\geq 1} \lambda_j  \| \Phi_j\|_{L^2(\T,\H)}
     \:=\: \sum_{j\geq 1} \lambda_j    \int_{\T} \| \Phi_j(u)\|^2_{\H}du  
     \:=\:    \int_{\T}  \sum_{j\geq 1} \lambda_j \| \Phi_j(u)\|^2_{\H}du  \\
     \:=&\:    \int_{\T}  \sum_{i,j\geq 1} \lambda_j \langle  \Phi_j(u), e_i \rangle_{\H} du  
     \:=\:    \int_{\T} \tr\left\{  \sum_{j\geq 1} \lambda_j  \Phi_j(u) \otimes_{\H} \Phi_j(u) \right\} du  \\
     \:=&\:    \int_{\T} \tr\{  K(u,u) \}du 
\end{align*}
where we have employed Parseval's equality, and have exchanged of the order of summation and integration by Fubini's Theorem.
\end{proof}

 \section{Karhunen–Loève Theorem for Hilbertian Flows}
Let $\{ \chi(t)\;:\; t\in\T\}$ be a stochastic process on a probability space $(\Omega,\mathscr{F},\PP)$, taking values in $\H$.
We may define its mean  $m\;:\; \T \longrightarrow \H$ by:
\begin{equation}\label{eq : mean}
     m(t) = \E[\chi(t)],  \qquad t\in\T
\end{equation}
and its {covariance kernel} $K\;:\; \T\times\T \rightarrow \BB(\H)$ by:
\begin{equation}\label{eq : cov}
    K(s,t) =\E [ \chi(s)\otimes_{\H} \chi(t) ],  \qquad t,s\in \T
\end{equation}
provided the above expectations are well defined as Bochner integrals (see \citet[][Definition 7.2.1]{hsing2015theoretical}.

We say that a process $\{ \chi(t)\;:\; t\in\T\}$ is a \textit{second-order process} if \eqref{eq : mean} and \eqref{eq : cov} are well defined for every $t\in\T$. Note that the covariance kernel is a symmetric operator-valued kernel, i.e.:
$$ 
K(s,t)= K(t,s)^*
$$

 We say that $\chi(\cdot)$ is \textit{mean-square continuous} if:
    $$
    \lim_{n\rightarrow\infty}  \E \left[\| \chi(t_n) - \chi(t)\|^2_{\H}\right]  \rightarrow 0 
    $$
for any $t\in\T$ and any sequence $\{t_n\}_{n\geq 1}$ converging to $t$.

\medskip

\begin{lemma}
Let $\{ \chi(t)\;:\; t\in\T\}$ be a second-order process. Then $\chi$ is mean-square continuous if and only if its mean and covariance functions are continuous wrt $\|\cdot\|_\H$ and $\|\cdot\|_{\BB_1(\H)}$, respectively.
\end{lemma}
\begin{proof}
Assume without loss of generality that the process is centered, i.e.\ that $m = 0$.
\begin{align*}
     E[ \|\chi(s)-\chi(t)\|_{\H}^2] \quad
     =& \E \| (\chi(s) - \chi(t))\otimes ( \chi(s) - \chi(t))\|_{\BB_1(\H)} \\
     = & \quad \| \E [ (\chi(s) - \chi(t))\otimes ( \chi(s) - \chi(t))] \|_{\BB_1(\H)} \\
     = & \quad \| \E [ (\chi(s) - \chi(t))\otimes \chi(s) - (\chi(s) - \chi(t))\otimes \chi(t))] \|_{\BB_1(\H)} \\
      = & \quad \| \E [ \chi(s) \otimes \chi(s)  - \chi(t)\otimes \chi(s) 
      - \chi(s)\otimes \chi(t) + \chi(t))\otimes \chi(t)] \|_{\BB_1(\H)} \\ 
      = & \quad  \| K(t,t) + K(s,s) - K(t,s) - K(t,s)\|_{\BB_1(\H)} \rightarrow 0, \quad \text{as } s\rightarrow t.
\end{align*}

Note that the kernel $K$ is non-negative definite, in the sense of \eqref{eq : positive kernel}. Indeed, for $n\geq 1$ and any sequences $\{v_j\}_{j=1,\dots,n}\subset \T$, $\{f_j\}_{j=1,\dots,n} \subset \H$:
\begin{align*}
\sum_{i,j=1}^n \langle K(v_i,v_j) f_i, f_j\rangle_{\H} 
\:  = \: \E \left[ \sum_{i,j=1}^n \langle \chi_{v_i}, f_i\rangle_{\H} \cdot \langle \chi_{v_j}, f_j\rangle_{\H} \right]
\:  = \:  \left( \sum_{j=1}^n \E  \langle \chi_{v_j}, f_j\rangle_{\H}\right)^2 \geq 0
\end{align*}
\end{proof}

\noindent We may finally state:

\begin{theorem}[Karhunen-Loève Expansion for Hilbertian Flows/Fields]
\label{kl-ours}
Let $\{ \chi(t)\;:\; t\in\T\}$ be a mean-squared continuous second-order process. Define:
$$
\chi_n(t) := \sum_{j=1}^n 
\langle\chi, \Phi_j\rangle_{L^2(\T,\H)}
\Phi_j(t),\qquad t\in \T.
$$
Then:
$$
\lim_{n\rightarrow\infty}\sup_{t\in \T}
\mathbb{E}\left[\| \chi_n(t) - \chi(t) \|_{\H}^2\right] = 0.
$$
\end{theorem}
\begin{proof}

First, note that:
\begin{align*}
    \E  \| \, \chi_n(t) \, \|_{\H }^2 
    \:&=\:     \sum_{i,j = 1}^n \langle \Phi_j(t), \Phi_i(t) \rangle_\H 
    \; \E \left [ \langle\chi, \Phi_j\rangle_{L^2(\T,\H)} \cdot
    \langle\chi, \Phi_i\rangle_{L^2(\T,\H)}  \right] \\
    \:&=\: \sum_{i,j = 1}^n \langle \Phi_j(t), \Phi_i(t) \rangle_\H 
    \; \langle \A_K \Phi_j,\Phi_i\rangle _{L^2(\T,\H)} \\
    \:&=\: \sum_{i,j = 1}^n \langle \Phi_j(t), \Phi_i(t) \rangle_\H 
    \; \lambda_j \delta_{i,j} 
\:=\: \sum_{j = 1}^n   \lambda_j  \|  \Phi_j(t)\|^2_{\H} 
\end{align*}
Furthermore:
\begin{align*}
    \E  \left [ \langle  \chi_n(t), \chi(t)\rangle _{\H } \right]
    \:&=\:   \sum_{j=1 }^n  \E  \left [ \langle\chi, \Phi_j\rangle_{L^2(\T,\H)} \cdot  \langle  \Phi_j(t), \chi(t)\rangle _{\H } \right] \\
    \:&=\:   \sum_{j=1 }^n \int_{\T} \E  \left [ \langle\chi(u), \Phi_j(u)\rangle_{\H} \cdot  \langle  \Phi_j(t), \chi(t)\rangle _{\H } \right] du\\
     \:&=\:   \sum_{j=1 }^n \int_{\T} \E \langle \: (\chi(u)\otimes_{\H} \chi(t))\Phi_j(u), \Phi_j(t) \: \rangle_{\H} \\
     \:&=\:   \sum_{j=1 }^n \int_{\T} \langle \: K(u,t) \Phi_j(u), \Phi_j(t) \: \rangle_{\H} \\
     \:&=\:   \sum_{i=1}^{\infty}  \sum_{j=1 }^n \int_{\T} \langle \: 
     \lambda_i \Phi_i(u)\otimes_{\H} \Phi_i(t)
     \Phi_j(u), \Phi_j(t) \: \rangle_{\H} du\\
    \:&=\:   \sum_{i=1}^{\infty}  \sum_{j=1 }^n \int_{\T} \langle \: 
     \lambda_i \langle \Phi_i(u), 
     \Phi_j(u) \rangle_{\H} \cdot \langle \Phi_i(t), \Phi_j(t)  \rangle_{\H} du\\
    \:&=\:   \sum_{i=1}^{\infty}  \sum_{j=1 }^n  \langle \: 
     \lambda_i \langle \Phi_i, 
     \Phi_j \rangle_{L^2(\T,\H)} \cdot \langle \Phi_i(t), \Phi_j(t)  \rangle_{\H} 
     \:=\:    \sum_{j=1 }^n  
     \lambda_j  \|  \Phi_j(t)\|^2_{\H} 
\end{align*}
where we have employed Fubini's theorem to interchange integral and expectation, and our extension to Mercer's decomposition, Theorem \ref{thm : mercers, ours} to express $K(u,t)$ as  $\sum_{i=1}^{\infty}  \lambda_i \langle \Phi_i(u) \otimes_{\H} \Phi_i(t) \rangle_{\H}$.
     
Putting things together, we obtain:
\begin{equation}
\label{eq KL final}
    \E \left[ \| \chi_n(t) - \chi \|_{\H }^2\right] 
    \:=\:  \E \left[ \| \chi_n(t)\|^2  + \| \chi \|_{\H }^2 - 2 \langle \chi_n, \chi \rangle_{\H}\right] 
    \:=\:  \tr\big\{ K(t,t)\big\} - \sum_{j=1 }^n  
   \lambda_j  \|  \Phi_j(t)\|^2_{\H} 
\end{equation}
and observing that, by Parseval's identity:
\begin{align*}
\tr\left\{ \sum_{j=1}^{n}  \lambda_j \langle \Phi_j(t) \otimes_{\H} \Phi_j(t) \rangle_{\H} \right\}
 \:=\:   \sum_{j=1}^{n} \sum_{i\geq 1}   \lambda_j \langle \Phi_j(t), e_i \rangle_{\H}^2
 \:=\: \sum_{j=1}^{n}    \lambda_j  \|  \Phi_j(t) \|^2_{\H}
\end{align*}
 we finally see that \eqref{eq KL final} converges to zero uniformly by our extension to Mercer's decomposition, Theorem \ref{thm : mercers, ours}, thus proving the theorem's statement.
\end{proof}

\section{Remarks on Computation}

We now briefly remark on the issue of computing the expansion components for a collection of realised flows. 

Let $\chi^{1}, \dots, \chi^{N}$ be independent realizations of $\chi$. Let $\T_{n} = \{t_{k}\}_{k=1}^{n}$ be a discretization of $\T$. Finally, let $\H_{m}$ be the $m$-dimensional subspace  of $\H$ spanned by the first $m$ vectors of some CONS $\{e_{k}\}_{k=1}^{\infty}$ of $\H$. Assume that at each of the nodes $\{t_k\}_{k=1}^{n}$, we can observe linear measurements $\langle \chi^j(t_k),e_i\rangle$ for each flow $\{\chi_j\}_{j=1}^N$ and each basis element $\{e_i\}_{i=1}^{m}$.

Under this measurement scheme, the realization $\chi^{j}$, $1\leq j \leq N$, is reduced to the column 
$$\mathbf{X}_{j} = [ \langle \chi^{j}(t_1), e_{1} \rangle, \langle \chi^{j}(t_{1}), e_{2} \rangle, \dots, \langle \chi^{j}(t_1), e_{m} \rangle, \langle \chi^{j}(t_2), e_{1} \rangle, \dots, \langle \chi^{j}(t_n), e_{m} \rangle ]^{\top}$$
The complete collection of observations can be concisely summarized by the $mn \times N$ matrix $\mathbf{X} = [\mathbf{X}_{1} ~ \mathbf{X}_{2} ~ \cdots ~ \mathbf{X}_{N}]$. A naive approach to computing the eigendecomposition $\{(\lambda_{j}, \Phi_{j})\}_{j=1}^{\infty}$ associated with $\chi$ would proceed by evaluating a discretized version of the empirical covariance and then computing the eigendecomposition of the resulting matrix. In our case, this corresponds to calculating the eigendecompostion of $\mathbf{K} = \tfrac{1}{N}\mathbf{X}\mathbf{X}^{\top}$. The computational complexity of this operation is $O(m^{3}n^{3} + m^{2}n^{2}N)$ ($O(m^3n^3)$ for the eigendecomposition and the rest for evaluating the covariance), which severely limits the resolution ($n$ and $m$).

Fortunately, by a classical trick \citep{chambers1977, Joiliffe1986}, we can circumvent this computational cost by simply calculating instead the singular value decomposition of $\mathbf{X}$. In practice, $mn > N$ and this allows us to write the SVD of $\mathbf{X} = \mathbf{U}\mathbf{D}\mathbf{V}$ where $\mathbf{U} = [u_{ij}]$ and $\mathbf{V}$ are orthogonal matrices of dimensions $mn \times N$ and $N \times mn$ and $\mathbf{D} = [d_{ij}\delta_{ij}]$ is an $N \times N$ diagonal matrix. For $k = 1, \dots, n$, we have the approximations
$$\hat{\Phi}_{j}(t_{k}) = \sqrt{n}\sum_{l=k+1}^{k+m} u_{lj} e_{l} \quad \mbox{ and } \quad \hat{\lambda}_{j} = d_{jj}^{2}/N$$
Compared to the complexity $O(m^{3}n^{3} + m^{2}n^{2}N)$ of the naive approach, this approximation's complexity is $O(\max(mn, N)\min(mn, N)) = O(mnN^{2})$. 

\hfill

\section{Discussion}

  \cite{kim2020principal} have also considered Hilbertian functional data, and we conclude by discussing the differences between the two contexts.   \cite{kim2020principal} obtain an optimal decomposition by optimizing over representations of the form
\begin{eqnarray*}
    \chi(t) = \sum_{j=1}^{\infty} \alpha_{j}(t) \Phi_{j}
\end{eqnarray*}
where $\{\alpha_{j}\} \subset L^{2}(\R)$ is the CONS to be optimized over and $\{\Phi_{j}\}_{j=1}^{\infty} \subset \H$ are random ``coefficients'' in the ambient Hilbert space. In this respect, it is quite different from the original Karhunen-Lo\`{e}ve expansion, where the role of coefficients and CONS is the ``reverse'': traditionally, the CONS lives in the same ambient space as the process and the coefficients are real-valued. The expansion of   \cite{kim2020principal} is arrived at via the eigendecomposition of the \textit{real-valued} autocovariance kernel $C(s,t) = \E \langle\chi(s),\chi(t)\rangle_{\H}$, for $s,t\in\T$. Consequently, their analysis produces eigenfunctions residing in $L^2(\T,\R)$, regardless of the nature of $\H$. 

In contrast, our decomposition is obtained by optimizing over representations of the form
\begin{eqnarray*}
    \chi(t) = \sum_{j=1}^{\infty} \alpha_{j} \Phi_{j}(t)
\end{eqnarray*}
where $\{\Phi_{j}\}_{j=1}^{\infty} \subset L^{2}(\T, \H)$ is the CONS to be optimized over and $\{\alpha_{j}\} \subset \R$ are the random real-valued coefficients. This stems from the spectral analysis of the \textit{operator-valued} kernel $K(s,t)=\E \chi(s)\otimes\chi(t)$, and consequently produces principal components living in the ambient space $L^2(\T,\H)$. These are arguably more natural, but in any sense compatible with the traditional Karhunen-Lo\`eve theorem.

\newpage

\bibliographystyle{chicago}
\bibliography{bib}

\end{document}